\documentclass[12pt]{amsart}
\usepackage[utf8]{inputenc}
\usepackage{amssymb, amsmath}
\usepackage{fourier}
\usepackage{enumerate}

\usepackage{geometry}
\theoremstyle{plain}

\newtheorem{theorem}{Theorem}

\newtheorem{claim}[theorem]{Claim}

\newtheorem{definition}[theorem]{Definition}

\theoremstyle{remark}

\newtheorem{example}{Example}

\newcommand{\bit}{\begin{itemize}}
\newcommand{\eit}{\end{itemize}}
\newcommand{\ben}{\begin{enumerate}}
\newcommand{\een}{\end{enumerate}}
\newcommand{\be}{\begin{equation}}
\newcommand{\ee}{\end{equation}}
\newcommand{\ba}{\begin{array}}
\newcommand{\ea}{\end{array}}

\newcommand{\mc}[1]{\mathcal{#1}}

\newcommand{\abs}[1]{\left|#1\right|}
\newcommand{\norm}[1]{\left|\left|#1\right|\right|}

\newcommand{\dd}{\mathrm{d}}

\newcommand{\eps}{\varepsilon}

\newcommand\cA{\mathcal A}
\newcommand\cB{\mathcal B}

\newcommand\R{\mathbb R}

\newcommand{\ler}[1]{\left( #1 \right)}

\title[Isometries on probability measures]{Maps on probability measures preserving certain distances --- a survey and some new results}

\author{D\'aniel Virosztek}
\address{Institute of Science and Technology Austria \\
Am Campus 1, 3400 Klosterneuburg, Austria}
\email{daniel.virosztek@ist.ac.at}
\urladdr{http://pub.ist.ac.at/\~{}dviroszt}

\dedicatory{Dedicated to the memory of Professor D\'enes Petz}
\thanks{The author was supported by the ISTFELLOW program of the Institute of Science and Technology Austria (project code IC1027FELL01) and partially supported by the Hungarian National Research, Development and Innovation Office – NKFIH (grant no. K124152).}


\begin{document}

\begin{abstract}
Borel probability measures living on metric spaces are fundamental mathematical objects. There are several meaningful distance functions that make the collection of the probability measures living on a certain space a metric space. We are interested in the description of the structure of the isometries of such metric spaces. We overview some of the recent results of the topic and we also provide some new ones concerning the Wasserstein distance. More specifically, we consider the space of all Borel probability measures on the unit sphere of a Euclidean space endowed with the Wasserstein metric $W_p$ for arbitrary $p \geq 1,$ and we show that the action of a Wasserstein isometry on the set of the Dirac measures is induced by an isometry of the underlying unit sphere.
\end{abstract}
\maketitle

\section{Introduction}
The study of isometries of various metric spaces has a huge literature. Some results that describe the structure of the isometries of some highly important spaces are very well-known. From our viewpoint, the most interesting classical result is the \emph{Banach-Stone theorem} that describes the surjective linear isometries between the function spaces $C(X)$ and $C(Y),$ where $X$ and $Y$ are compact Hausdorff spaces. The Banach-Stone theorem says that every such isometry is the composition of an isometry induced by a homeomorphism between the underlying spaces $X$ and $Y$ and a trivial isometry. (We will make this statement precise later.)
Another example of the well-know classical results on isometries is the \emph{Mazur-Ulam theorem} which states that a surjective isometry between real normed spaces is necessarily affine. For a comprehensive study of isometries, moreover, other types of preserver problems, we refer to the monographs \cite{fleming-jamison-1, fleming-jamison-2, molnar-book}.
\par
Isometries of spaces of measures (or distribution functions) have also been studied extensively. In a series of papers, \emph{Lajos Moln\'ar} (partially with \emph{Gregor Dolinar}) described the isometries of the distribution functions with respect to the \emph{Kolmogorov-Smirnov metric} and the \emph{L\'evy metric} (see \cite{dolinar-molnar,molnar-kolm,molnar-levy}). As a substantial generalization of Moln\'ar's result on the L\'evy isometries, \emph{Gy\"orgy P\'al Geh\'er} and \emph{Tam\'as Titkos} managed to describe the surjective isometries of the space of all Borel probability measures on a separable real Banach space with respect to the \emph{L\'evy-Prokhorov distance} \cite{geher-titkos}. Geh\'er also described the surjective isometries of the probability measures on the real line with respect to the \emph{Kuiper metric} \cite{geher-kuiper}. The \emph{Wasserstein isometries} have been investigated by \emph{J\'erome Bertrand and Benoit R. Kloeckner} on various spaces with the special choice of the parameter $p=2$ \cite{Kloeckner-2008,bertrand-kloeckner-2016}. \par
Our goal is to study the Wasserstein isometries on probability measures defined on unit spheres for an arbitrary parameter $p \geq 1.$ We make some progress in the direction of a Banach-Stone-type result, that is, we show that the action of a Wasserstein isometry on the set of the Dirac measures is induced by an isometry of the underlying unit sphere.   
\section{Optimal transport}
\subsection{Motivation}
Let us consider the following problem. There are $m$ producers of a certain product, say, $x_1, \dots, x_m,$ and there are $n$ customers which are denoted by $y_1, \dots, y_n.$ The producer $x_i$ offers $p_i$ unit of the product and the customer $y_j$ needs $q_j$ unit of it. Assume that we are in the fortunate situation when the total demand coincides with the total supply, that is, $\sum_{i=1}^m p_i=\sum_{j=1}^n q_j.$ For the sake of simplicity, we assume that the aforementioned quantities are equal to $1.$ Let us denote the cost of transferring a unit of product from $x_i$ to $y_j$ by $c(i,j).$ A {\it transference plan} (or {\it transport plan}) is a declaration of the amounts of the product that are to be transferred from the sources to the targets. Let $t(i,j)$ denote the amount that is to be transferred from $x_i$ to $y_j.$ Then a transference plan is an array of nonnegative real numbers $\left\{ t(i,j)\right\}_{i=1, \, j=1}^{m \quad n}$ such that $\sum_{j=1}^n t(i,j)=p_i$ and $\sum_{i=1}^m t(i,j)=q_j$ for all $i$ and $j.$
\par
We are interested in finding the minimal cost of transferring the product from the producers to the customers. Clearly, the minimal cost is
\be \label{eq:disc-min-cost}
\mathrm{inf} \, \sum_{i,j} c(i,j) t(i,j)
\ee
where the infimum runs over all transport plans. The quantity \eqref{eq:disc-min-cost} is called the {\it optimal transport cost} between the probability measures $\left\{p_i\right\}_{i=1}^m$ and $\left\{q_j\right\}_{j=1}^n.$

\subsection{The mathematical treatment of more general optimal transport problems}
The optimal transport cost may be defined between any Borel probability measures on sufficiently nice spaces. The key notion which is needed to define optimal transport cost between general probability measures is the \emph{coupling,} which is a basic concept in probability theory with a lot of applications that are different from optimal transport (see, e.g., \cite[Chapter 1]{villani-book}).

\begin{definition}[Coupling] \label{defi:coupling}
Let $X$ and $Y$ be Polish (that is, separable and complete) metric spaces and let $\mu$ and $\nu$ be Borel probability measures on $X$ and $Y,$ respectively. A Borel probability measure $\pi$ on $X \times Y$ is said to be a \emph{coupling} of $\mu$ and $\nu$ if the marginals of $\pi$ are $\mu$ and $\nu,$ that is, $\pi\ler{A \times Y}=\mu(A)$ and $\pi\ler{X \times B}=\nu(B)$ for all Borel sets $A\subset X$ and $B \subset Y.$
\end{definition}

Let us denote the set of all couplings of the probability measures $\mu$ and $\nu$ by $\Pi\ler{\mu, \nu}.$ Now, let $c(x,y)$ stand for the cost of transporting one unit of mass from $x \in X$ to $y \in Y.$ (In this contex, the word "mass" refers to something that is to be transferred.)
The optimal transport cost between the measures $\mu$ and $\nu$ is defined as
\be \label{eq:opt-tr-cost-def}
C\ler{\mu, \nu}:=\inf_{\pi \in \Pi\ler{\mu, \nu}}  \int_{X \times Y} c (x,y) \mathrm{d}\pi(x,y).
\ee
\subsection{Metric properties of the optimal transport cost}
One may expect that the quantity \eqref{eq:opt-tr-cost-def} serves as a distance between probability measures. In general, $\ler{\mu,\nu}\mapsto C\ler{\mu, \nu}$ is not a metric, but there are some important special cases when it is indeed a metric. When the cost function is defined in terms of a metric appropriately, then the optimal transport cost is (in a very simple correspondence with) a metric on measures. The \emph{Wasserstein distances} are metrics on measures that are defined as very simple functions of optimal transport costs induced by special cost functions. In order to define Wasserstein distances, first we need to define \emph{Wasserstein spaces}. Here and throughout, let $P(X)$ denote the set of all Borel probability measures on a metric space $X.$

\begin{definition}[Wasserstein spaces] \label{defi:wass-space}
Let $(X,d)$ be Polish metric space and let $1 \leq p <\infty.$ The \emph{Wasserstein space of order $p$} is defined as
$$
P_p(X):=\left\{ \mu \in P(X) \, \middle| \, \int_X d(x_0,x)^p \dd \mu(x)< \infty \text{ for some (hence all) } x_0 \in X\right\}.
$$
\end{definition}
In words, the Wasserstein space of order $p$ consists of the probability distributions that have finite moment of order $p.$ Clearly, if the metric $d$ is bounded on $X,$ then we have $P_p(X)=P(X)$ for all $p \in [1, \infty).$
\par
Now we are in the position to define the Wasserstein distances.

\begin{definition}[Wasserstein distances] \label{defi:wass-dist}
With the same conventions as in Definition \ref{defi:wass-space}, the \emph{Wasserstein distance of order $p$} between $\mu \in P_p(X)$ and $\nu \in P_p(X)$ is defined by the formula
\be \label{eq:wasser_def}
W_p\ler{\mu, \nu}:=\ler{\inf_{\pi \in \Pi(\mu, \nu)} \int_{X \times X} d (x,y)^p \mathrm{d}\pi(x,y)}^{\frac{1}{p}}.
\ee
\end{definition}
It can be shown that the Wasserstein distance of order $p$ (or $p$-Wasserstein distance) is a true metric on $P_p(X)$ (see, e.g., \cite[Chapter 6]{villani-book}, or \cite{dobrushin} for the special case $p=1$). The Wasserstein distances encode valuable geometric information as they are defined in terms of the underlying geometry. In particular we have $W_p \ler{\delta_x, \delta_y}=d(x,y)$ for any Polish space $(X,d)$ and any $x,y \in X$ and $1 \leq p<\infty.$ (Here and throughout, $\delta_x$ denotes the \emph{Dirac measure} concentrated on the point $x \in X.$) Consequently, the Polish space $X$ can be embedded isometrically into the measure space $P_p(X)$ by the map $x \mapsto \delta_x$ for any $1\leq p<\infty.$ Moreover, any isometry of $X$ induces a $p$-Wasserstein isometry on $P_p(X)$ (for any $p$) by the \emph{push-forward of measures}. This latter concept is of a particular importance in measure theory and it will play a crucial role throughout this paper, hence we define it in a quite general context.

\begin{definition}[Push-forward] \label{defi:push-forw}
Let $(X,\cA)$ and $(Y, \cB)$ be measurable spaces and let $\mu$ be complex measure on $X.$ Let $\psi: (X,\mathcal{A}) \rightarrow (Y, \cB)$ be a measurable map. Then the \emph{push-forward of the measure $\mu$} by the map $\psi$ is denoted by $\psi_\# \mu$ and it is defined by
$$
\psi_\# \mu(B):=\mu \ler{\psi^{-1}(B)} \qquad \ler{B \in \cB}.
$$

\end{definition}

Indeed, it is easy to see that for an isometry $\psi: X \rightarrow X$ the induced push-forward of measures $\psi_\# : P_p(X) \rightarrow P_p(X)$ is a $p$-Wasserstein isometry for any $p.$ So, we have a natural group homomorphism from the isometry group of $X$ into the isometry group of $P_p(X)$ which looks as follows:
\be \label{eq:hashtag}
\#: \, \mathrm{Isom} \, X \rightarrow \mathrm{Isom} \, P_p(X); \qquad \psi \mapsto \psi_\#.
\ee

\subsection{Some remarkable properties of the Wasserstein distance of order $1$}
In the sequel we recall two interesting properties of the distance $W_1$ (which is also commonly called the \emph{Kantorovich–Rubinstein distance}) on particular Polish metric spaces.

\begin{example} \label{ex:w1-tv}
The \emph{total variation distance} is a well known metric on $P(X)$ defined by the formula
\be \label{eq:tv-def}
d_{TV}\ler{\mu, \nu}=\sup_{B \in \cB_X}\abs{\mu(B)-\nu(B)},
\ee
where $\cB_X$ denotes the collection of all Borel sets of $X.$
\par
Let $(X,d)$ be a discrete metric space, that is, $X$ is a nonempty set and $d: X \times X \rightarrow [0,\infty)$ is defined by
$$
d(x,y)=
\begin{cases}
0, \text{ if } x=y, \\
1, \text{ if } x\neq y.
\end{cases}
$$
Then the total variation distance coincides with the Wasserstein distance of order $1$ on $P(X)=P_1(X)$ (see \cite{dobrushin} and \cite{vallender}).
\end{example}

\begin{example} \label{ex:w1-kant}
Let $X=\R$ equipped with the usual (Euclidean) metric. In this special case, the Wasserstein distance of order $1$ can be expressed explicitly in terms of the cumulative distribution functions by the formula
\be \label{eq:w1r-ekv}
W_1(\mu, \nu)=\int_\R \abs{F(x)-G(x)},
\ee
where $F(x)=\mu\ler{(-\infty,x]}$ and $G(x)=\nu\ler{(-\infty,x]}.$ This result is due to \emph{Vallender} \cite{vallender}.

\end{example}

\section{Isometries of measure spaces: an overview of the literature}

The study of isometries of measure spaces is an extensive topic in the area of preserver problems. Throughout this paper, by measure spaces we mean collections of Borel probability measures on Polish metric spaces. Different notions of distance lead to different geometry on measure spaces. In order to understand a geometric structure one has to face several challenges. One of the most fundamental characteristics of a geometric structure is its isometry group, so the description of the isometries belongs certainly to the important challenges.
\par
In the sequel we recall some results on isometries of measure spaces. Certainly, this enumeration of the relevant works is far from being complete. As the main result of this note is a step in the direction of a \emph{Banach-Stone-type result,} first we shall recall the famous Banach-Stone theorem.

\begin{theorem}[Banach-Stone] \label{thm:banach-stone}
Let $X$ and $Y$ be compact, Hausdorff topological spaces and let $C(X)$ and $C(Y)$ denote the spaces of all continuous complex-valued functions on $X$ and $Y,$ respectively (equipped with the supremum norm). Let $T: C(X) \rightarrow C(Y)$ be a \emph{surjective, linear} isometry. Then there exists a homeomorphism $\varphi: Y \rightarrow X$ and a function $u \in C(Y)$ with $\abs{u(y)}=1$ for all $y \in Y$
such that
$$
(T f) (y)=u(y) f \ler{\varphi(y)} \qquad \text{ for all } y \in Y \text{ and } f \in C(X).
$$
\end{theorem}
Seemingly, this classical result does not have any connection with measures.
\par
Let us remark that the Banach-Stone theorem describes the structure of the surjective linear isometries between unital \emph{commutative} $C^*$-algebras. (By the \emph{Gelfand-Naimark theorem,} any such algebra is isometrically $*$-isomorphic to $C(K)$ for some compact Hausdorff space $K;$ the $*$ operation on $C(K)$ is the pointwise conjugation, that is, $f^*(k)=\overline{f(k)}$ for all $k \in K.$) It states that any surjective linear isometry is necessarily an algebra $*$-isomorphism --- up to multiplication by a fixed function of modulus $1.$
The \emph{Kadison theorem} is a generalization of the Banach-Stone theorem for \emph{not necessarily commutative} unital $C^*$-algebras. It says that a surjective linear isometry between unital $C^*$-algebras can be obtained as a \emph{Jordan $*$-isomorphism} multiplied by a fixed unitary element. (A Jordan $*$-isomorphism is a bijective linear map $J$ that respects the $*$ operation and preserves the square, that is, $J\ler{a^2}=J(a)^2$ for all $a.$)
\par
There are several results in the large area of preserver problems which state that "any isometry between certain extra structures built on sets (say, function spaces, measure spaces, etc.) is necessarily driven by some sufficiently nice transformation between the underlying sets". Such results are called Banach-Stone-type theorems for obvious reasons.

\subsection{Banach-Stone-type results on isometries of measure spaces}
In this subsection we recall some recent Banach-Stone-type results concerning measure spaces.

\subsubsection{Kolmogorov-Smirnov isometries}
The first non-classical result that we recall here is the theorem of Dolinar and Moln\'ar on the isometries of the space of probability distributions on the real line with respect to the \emph{Kolmogorov-Smirnov metric}. The Kolmogorov-Smirnov distance of the Borel probability measures $\mu, \nu \in P(\R)$ is defined by the formula
\be \label{eq:KS-def}
d_{KS}\ler{\mu, \nu}:=\norm{F_\mu-F_\nu}_\infty=\sup_{x \in \R}\abs{F_\mu(x)-F_\nu(x)},
\ee
where $F_\eta$ stands for the cumulative distribution function of the measure $\eta$ for any $\eta \in P(\R),$ that is, $F_\eta(x)=\eta\ler{(-\infty, x]}.$
\par
Note that the Kolmogorov-Smirnov distance is closely related to the  total variation distance (introduced in Example \ref{ex:w1-tv}) and the $1$-Wasserstein distance on $P_1(\R)$ (see Example \ref{ex:w1-kant}). It is clear by the comparison of the formulas \eqref{eq:tv-def} and \eqref{eq:KS-def} that $d_{KS}\ler{\mu, \nu} \leq d_{TV} \ler{\mu, \nu}$ always holds, and by the comparison of the formulas \eqref{eq:w1r-ekv} and \eqref{eq:KS-def} one may observe that the $1$-Wasserstein distance is just the $L_1$ distance of the distribution functions while the Kolmogorov-Smirnov distance is the $L_\infty$ distance of them. The theorem of Dolinar and Moln\'ar reads as follows.
\begin{theorem}[\cite{dolinar-molnar}] \label{thm:dol-moln}
Let $\phi: P(\R) \rightarrow P(\R)$ be a surjective Kolmogorov-Smirnov isometry, that is, a bijection on $P(\R)$ with the property that
$$
d_{KS}\ler{\phi(\mu),\phi(\nu)}=d_{KS} \ler{\mu, \nu} \qquad \ler{\mu, \nu \in P(\R)}.
$$
Then either there exists a strictly increasing bijection $\psi: \R \rightarrow \R$ such that
\be \label{eq:ks-opp-1}
F_{\phi(\mu)}(t)=F_\mu\ler{\psi(t)} \qquad \ler{t \in \R, \mu \in P(\R)},
\ee
or there exits a strictly decreasing bijection $\tilde{\psi}: \R \rightarrow \R$ such that
\be \label{eq:ks-opp-2}
F_{\phi\ler{\mu}}(t)=1-F_{\mu}\ler{\tilde{\psi}(t)-} \qquad \ler{t \in \R, \mu \in P(\R)},
\ee
where $F_\eta(x-)$ denotes the left limit of the distribution function $F_\eta$ at the point $x$ --- note that $F_\eta(x-)=\eta\ler{\ler{-\infty,x}}.$ 
\par
Moreover, any transformation of the form \eqref{eq:ks-opp-1} or \eqref{eq:ks-opp-2} is a surjective Kolmogorov-Smirnov isometry.
\end{theorem}

Although Theorem \ref{thm:dol-moln} is formulated in terms of distribution functions, it can be easily reformulated in terms of measures as follows: for any surjective Kolmogorov-Smirnov isometry $\phi: P(\R) \rightarrow P(\R)$ there is a homeomorphism $\varphi: \R \rightarrow \R$ such that 
$$
\phi(\mu)=\varphi_\# \mu \qquad \ler{\mu \in P(\R)},
$$
where $\varphi_\#$ is the push-forward induced by $\varphi$ (see Definition \ref{defi:push-forw}). Indeed, if $\phi$ acts on $P(\R)$ such that \eqref{eq:ks-opp-1} holds, then $\varphi=\psi^{-1},$ that is, $\phi=\ler{\psi^{-1}}_\#$ and if $\phi$ acts on $P(\R)$ such that \eqref{eq:ks-opp-2} holds, then then $\varphi=\ler{\tilde{\psi}}^{-1},$ that is, $\phi=\ler{\ler{\tilde{\psi}}^{-1}}_\#.$
\par
The key idea of the result of Dolinar and Moln\'ar is the observation that the Dirac distributions can be characterized in terms of the Kolmogorov-Smirnov metric. To precisely state the characterization we shall introduce the following notation: for a metric space $(Y, \rho)$ and a set $S \subset Y$ let $U(S)$ be defined by
$$
U(S):=\left\{y \in Y \middle| \rho\ler{y, s}=1 \text{ for all } s \in S\right\}.
$$
Now let us consider the special metric space $\ler{P(\R),d_{KS}}.$
The metric characterization of the trivial distributions reads as follows.
\be \label{eq:bidual}
\text{A measure }\mu \in P(\R) \text{ is a Dirac mass }  \Longleftrightarrow U\ler{U\ler{\left\{\mu\right\}}}=\left\{\mu\right\}.
\ee
Such characterizations of Dirac measures will play a crucial role in several following results.
\subsubsection{L\'evy isometries}
The L\'evy distance of the measures $\mu, \nu \in P(\R)$ is defined as follows:
$$
d_{LE} \ler{\mu, \nu}
$$
$$
=\inf \left\{ \varepsilon >0\,\middle|\, \mu\ler{(-\infty, t-\varepsilon]}-\varepsilon \leq \nu\ler{(-\infty, t]} \leq \mu \ler{(-\infty, t+\varepsilon]}+\varepsilon \text{ for all } t \in \R\right\}.
$$
Let us remark that the equivalent definition
$$
d_{LE} \ler{\mu, \nu}
$$
$$
= \sup \left\{ \varepsilon >0\,\middle|\, \nu\ler{(-\infty, t]}+\varepsilon < \mu\ler{(-\infty, t-\varepsilon]} \right.
$$
$$
\left. \text{ or } \nu\ler{(-\infty, t]}-\varepsilon > \mu\ler{(-\infty, t+\varepsilon]} \text{ for some } t \in \R\right\}
$$
offers another viewpoint to understand the L\'evy metric. The importance of the L\'evy distance comes from the fact that (just like some other metrics) it metrizes the topology of weak convergence in $P(\R).$ This type of convergence is of a particular importance in probability theory.
\par
Moln\'ar's theorem reads as follows.

\begin{theorem}[\cite{molnar-levy}] \label{thm:mol-levy}
Let $\phi: P(\R) \rightarrow P(\R)$ be a surjective L\'evy isometry, that is, a bijection on $P(\R)$ with the property that
$$
d_{LE}\ler{\phi(\mu),\phi(\nu)}=d_{LE} \ler{\mu, \nu} \qquad \ler{\mu, \nu \in P(\R)}.
$$
Then there is a constant $c \in \R$ such that either
\be \label{eq:le-iso-opp-1}
F_{\phi(\mu)}(t)=F_\mu\ler{t+c} \qquad \ler{t \in \R, \mu \in P(\R)}
\ee
or 
\be \label{eq:le-iso-opp-2}
F_{\phi\ler{\mu}}(t)=1-F_{\mu}\ler{(-t+c)-} \qquad \ler{t \in \R, \mu \in P(\R)}
\ee
holds.
\par
Moreover, any transformation of any of the forms \eqref{eq:le-iso-opp-1}, \eqref{eq:le-iso-opp-2} is a surjective L\'evy isometry on $P(\R).$
\end{theorem}

The easy part of Theorem \ref{thm:mol-levy} says that, similarly to the Wasserstein distances, the L\'evy metric has the property that the map $\psi \mapsto \psi_\#$ is a group homomorphism from $\mathrm{Isom} \, \R$ into $\mathrm{Isom} \, P(\R)$ (where the latter group consists of all surjective isometries of $P(\R)$ with respect to the L\'evy metric). The difficult part of Theorem \ref{thm:mol-levy} says that this group homomorphism is in fact onto, hence a group isomorphism.

The key idea is a metric characterization of the Dirac distributions (see equation \eqref{eq:bidual}), similarly to the proof of the result in \cite{dolinar-molnar}.

\subsubsection{Kuiper isometries}

The Kuiper distance of the probability measures $\mu, \nu \in P(\R)$ is given by the formula
\be \label{eq:kui-dist-def}
d_{KU}\ler{\mu, \nu}:=\sup_{I \in \mathcal{I}} \abs{\mu(I)-\nu(I)},
\ee
where $\mathcal{I}=\left\{ I \subset R \, \middle| \, \#I>1 \text{ and }I \text{ is connected}\right\},$ that is, $\mc{I}$ denotes the set of all non-degenerate intervals of $\R.$ It is clear from the definitions that the inequality
$$
0\leq d_{KS}\ler{\mu, \nu} \leq d_{KU}\ler{\mu,\nu} \leq d_{TV}\ler{\mu,\nu} \leq 1 \qquad \ler{\mu, \nu \in P(\R)}
$$
holds (compare the formula \eqref{eq:kui-dist-def} to the formulas \eqref{eq:tv-def} and \eqref{eq:KS-def}).
\par
The theorem of Geh\'er on the isometries of $P(\R)$ with respect to the Kuiper metric reads as follows.

\begin{theorem}[\cite{geher-kuiper}] \label{thm:geher-kuiper}
Let $\phi: P(\R) \rightarrow P(\R)$ be a surjective Kuiper isometry, that is, a bijection on $P(\R)$ with the property that
$$
d_{KU}\ler{\phi(\mu), \phi(\mu)}=d_{KU}\ler{\mu, \nu} \qquad \ler{\mu, \nu \in P(\R)}.
$$
Then there exists a homeomorphism $g: \R \rightarrow \R$ such that
$$
\phi(\mu)=g_\# \mu \qquad \ler{\mu \in P(\R)}.
$$
Moreover, every transformation of this form is a surjective Kuiper isometry on $P(\R).$
\end{theorem}
\subsubsection{L\'evy-Prokhorov isometries}

As mentioned before, the L\'evy distance is an important metric on $P(\R)$ as it metrizes the weak convergence in $P(\R).$
In $1956$ \emph{Prokhorov} introduced a metric which metrizes the weak convergence in $P(X)$ for a general Polish metric space $(X,d)$ \cite{prokhorov}. Now we call this metric \emph{L\'evy-Prokhorov distance} although it does not coincide with the L\'evy metric in the special case $X=\R.$ The L\'evy-Prokhorov distance is defined as follows:
$$
d_{LP}\ler{\mu, \nu}=\inf \left\{ \varepsilon>0 \, \middle| \, \mu(A) \leq \nu\ler{A^{\varepsilon}}+\varepsilon \text{ for all } A \in \cB_X\right\},
$$
where
$$
A^{\varepsilon}=\bigcup_{x \in A} B_\varepsilon(x) \text{ and } B_\varepsilon(x)=\left\{y \in X \,\middle|\, d(x,y)<\varepsilon \right\}.
$$
Geh\'er and Titkos considered the problem of determining the L\'evy-Prokhorov isometries of $P(X)$ in the case when $X$ is a separable real Banach space which is a bit less general setting than the setting of Polish metric spaces (which is the most general possible setting) \cite{geher-titkos}. Their result reads as follows.

\begin{theorem}[\cite{geher-titkos}] \label{thm:geh-titk}
Let $\ler{X,\norm{\cdot}}$ be a separable real Banach space and let $\phi: P(X) \rightarrow P(X)$ be a surjective L\'evy-Prokhorov isometry, that is, assume that
$$
d_{LP}\ler{\phi(\mu), \phi(\mu)}=d_{LP}\ler{\mu, \nu} \qquad \ler{\mu, \nu \in P(X)}
$$
holds. Then there exists a surjective affine isometry $\psi: X \rightarrow X$ which induces $\phi,$ that is, we have
\be \label{eq:lev-prok-thm}
\phi(\mu)=\psi_\# \mu \qquad \ler{\mu \in P(X)}.
\ee
Moreover, any transformation of the form \eqref{eq:lev-prok-thm} is a surjective L\'evy-Prokhorov isometry.
\end{theorem}
Similarly to the case of the L\'evy distance, we learned that the map $\psi \mapsto \psi_\#$ is a group homomorphism from $\mathrm{Isom} \, X$ into $\mathrm{Isom} \, P(X)$ (easy), and that this homomorphism is actually onto (difficult).
\par
The observation \eqref{eq:bidual} plays an important role in the proof of the result of Geh\'er and Titkos, as well. However, the general setting of separable real Banach spaces required the development of other involved techniques. 
\subsubsection{$2$-Wasserstein isometries on negatively curved spaces}

We have noted before that if we consider Wasserstein distances on $P_p(X)$ for a Polish space $X,$ then the push-forward of measures by an isometry of $X$ is always a $p$-Wasserstein isometry on $P_p(X),$ no matter what the value of the parameter $p$ is. (See equation \eqref{eq:hashtag}).
\par
The question naturally appears: are there isometries of $P_p(X)$ that can not be obtained this way? In other words: are there non-trivial isometries of $P_p(X)$? (Following the terminology of \cite{bertrand-kloeckner-2016}, we call a $p$-Wasserstein isometry $\phi$ of $P_p(X)$ \emph{trivial} if $\phi=\psi_\#$ for some isometry $\psi: X \rightarrow X.$ Moreover, an isometry $\phi$ is called \emph{shape-preserving} if for any $\mu \in P_p(X)$ there exists an isometry $\psi_\mu: X \rightarrow X$ such that $\phi (\mu)=\ler{\psi_\mu}_{\#}\mu.$ The isometries that are not even shape-preserving are called \emph{exotic} isometries.)
\par
The result of Bertrand and Kloeckner states that if $X$ is a negatively curved space, then all the $2$-Wasserstein isometries of $P_2(X)$ are trivial \cite{bertrand-kloeckner-2016}. In other words, the measure space $P_2(X)$ is \emph{isometrically rigid.}
\par
The precise statement reads as follows.

\begin{theorem}[\cite{bertrand-kloeckner-2016}] \label{thm:bert-kloeck}
Let $X$ be a negatively curved geodesically complete Hadamard space. Let $\phi: P_2(X) \rightarrow P_2(X)$ be a $2$-Wasserstein isometry, that is, assume that
$$
W_2 \ler{\phi(\mu), \phi(\mu)}=W_2 \ler{\mu, \nu} \qquad \ler{\mu, \nu \in P_2(X)}.
$$
Then there is an isometry $\psi: X \rightarrow X$ such that
$$
\phi(\mu)=\psi_\# \mu \qquad \ler{\mu \in P_2(X)}.
$$
\end{theorem}

Note that with the terminology borrowed from \cite{bertrand-kloeckner-2016} the results of \cite{molnar-levy} and \cite{geher-titkos} can be rephrased as follows: the measure space $P(\R)$ equipped with the L\'evy metric is isometrically rigid, and the measure space $P(X)$ for a real separable Banach space $X$ equipped with the L\'evy-Prokhorov metric is also isometrically rigid.  
\subsection{Non-Banach-Stone-type results on isometries of measure spaces}

Quite surprisingly, the probability measures on Euclidean spaces have non-trivial $2$-Wasserstein isometries, as well. Furthermore, in the special case of the real line we have also exotic isometries (recall that exotic means that it does not preserve the shape of the measures).
\par
The precise statements of Kloeckner about the isometries of $P_2$ spaces over Euclidean spaces read as follows.

\subsubsection{The case of the real line}
\begin{theorem}[\cite{Kloeckner-2008}] \label{thm:kloeck-euk}
The isometry group of the space $P_2(\R)$ with respect to the $2$-Wasserstein metric is a semidirect product
\be \label{eq:w2r-isom-gr}
\mathrm{Isom} \, \R \ltimes \mathrm{Isom} \, \R.
\ee
In \eqref{eq:w2r-isom-gr} the left factor is the image of $\#$ (recall that $\#$ was introduced in \eqref{eq:hashtag}) and the right factor consists of all isometries that fix pointwise the set of Dirac measures. Moreover, the right factor decomposes as $\mathrm{Isom} \, \R= C_2 \ltimes \R,$ where the $C_2$ factor (the group of order $2$) is generated by a non-trivial involution that preserve shapes and the $\R$ factor is a flow of exotic isometries. 
\end{theorem}

The question naturally appears: how do the elements of the right factor of \eqref{eq:w2r-isom-gr} look like? That is, how does a $2$-Wasserstein isometry that fixes all Dirac measures look like? The description of the exotic isometries is beyond the scope of this paper, we refer to the original work of Kloeckner \cite{Kloeckner-2008}. However, the description of the non-trivial but still shape-preserving isometries is easy; the reader will find it in the explanation of Theorem \ref{thm:kloeck-euk-n}, because the behavior of the shape-preserving isometries is independent of the dimension of the underlying Euclidean space. Keep in mind that $C_2=O(1).$

\subsubsection{The case of $\R^n$ for $n \geq 2$}

\begin{theorem}[\cite{Kloeckner-2008}] \label{thm:kloeck-euk-n}
For $n \geq 2,$ the $2$-Wasserstein isometry group of $P_2\ler{\R^n}$ is a semidirect product
\be \label{eq:iso-r-n}
\mathrm{Isom} \, \R^n \ltimes O(n)
\ee
where the action of an element $T \in \mathrm{Isom} \, \R^n$ on $O(n)$ is the conjugacy by its linear part $\tilde T.$
\par
The left factor in \eqref{eq:iso-r-n} is the image of $\#$ (see \eqref{eq:hashtag}) and each element of the right factor fixes all Dirac measures and preserves shapes.
\end{theorem}
So in "higher" dimensions we do not have exotic isometries but we still have non-trivial isometries.
We need some notation to explain how the non-trivial isometries look like.
\par
Given a $\mu \in P_2\ler{\R^n},$ the \emph{center of mass of $\mu$} is denoted by $c_\mu,$ that is, $c_\mu=\int_{\R^n} x \dd \mu(x).$ Furtheremore, for any $y \in \R^n,$ the associated translation is denoted by $\eta_y,$ that is, $\eta_y(x)=x+y \, \ler{x \in \R^n}.$ Now we can describe the non-trivial isometries: for any $\varphi \in O(n),$ the map
$$
\mu \mapsto \ler{\eta_{c_\mu}}_{\#}\circ \varphi_{\#} \circ \ler{\eta_{c_\mu}^{-1}}_{\#} (\mu)
$$
is a $2$-Wasserstein isometry which leaves the Dirac measures invariant.
\section{Wasserstein isometries on $P\ler{S^{n-1}}$}
After having reviewed some recent results in the topic, now we turn to the main problem of the current paper which is the description of the $p$-Wasserstein isometries on measures defined on unit balls of Euclidean spaces.
\par
Let $n \geq 2$ be arbitrary and let us consider the separable metric space
$$S^{n-1}:=\left\{x \in \R^n \, \middle| \, \norm{x}=\frac{1}{2}\right\},$$
where $\norm{.}$ denotes the Euclidean norm. We consider the Euclidean distance $d(x,y)=\norm{x-y}$ on the unit sphere $S^{n-1}.$ Clearly, the Euclidean distance is bounded on any unit ball, so for all $n\geq 2$ we have $P_p\ler{S^{n-1}}=P\ler{S^{n-1}}$ for all $p\geq 1.$ Our arguments that we present soon works for all $p \geq 1,$ so from now on, let $p\in [1, \infty)$ be arbitrary.

\begin{claim} \label{claim:dirac-char}
Set $\mu \in P\ler{S^{n-1}}.$ The followings are equivalent.
\ben
\item \label{mu-dirac} $\mu$ is a Dirac measure, that is, there exists $x \in S^{n-1}$ such that $\mu=\delta_x.$
\item \label{egytavolsag} There exists $\nu \in P\ler{S^{n-1}}$ such that $W_p(\mu, \nu)=1.$
\een
\end{claim}

\begin{proof}
The implication \eqref{mu-dirac} $\Longrightarrow$ \eqref{egytavolsag} is clear by the following short argument. For any $x,y \in S^{n-1}$ and for any $p\geq 1,$ we have $W_p \ler{\delta_x, \delta_y}=d(x,y)=\norm{x-y}.$ Therefore, if $\mu=\delta_x,$ then by the choice $\nu=\delta_{-x}$ we have
$$
W_p\ler{\mu, \nu}=W_p\ler{\delta_x, \delta_{-x}}=d(x,-x)=\norm{x-(-x)}=1.
$$
The proof of the direction \eqref{egytavolsag} $\Longrightarrow$ \eqref{mu-dirac} is a bit more complicated. We have to show that if $\mu \in P\ler{S^{n-1}}$ is not a Dirac measure, then $W_p(\mu, \nu)<1$ holds for all $\nu \in P\ler{S^{n-1}}.$ So, assume that $\mu \in P\ler{S^{n-1}}$ is not a Dirac measure and let $y \in S^{n-1}$ be arbitrary. Then there exists some $\eps >0$ such that
\be \label{eq:eta-def}
\mu \ler{\left\{x \in S^{n-1} \, \middle| \, d(x,y) \leq 1-\eps\right\}}>0.
\ee
(Otherwise, $\mu$ would be equal to $\delta_{-y}.$) Let us denote by $\eta$ this positive number appearing on the left hand side of \eqref{eq:eta-def} in the sequel. The estimation
$$
\int_{S^{n-1}} d(x,y)^p \dd \mu (x)
$$
$$
=
\int_{\left\{x \in S^{n-1}\,\middle|\, d(x,y) \leq 1-\eps\right\}}
d(x,y)^p \dd \mu (x)
+
\int_{\left\{x \in S^{n-1}\,\middle|\, d(x,y) > 1-\eps\right\}}
d(x,y)^p \dd \mu (x)
$$
$$
\leq \eta (1-\eps)^p+(1-\eta) \cdot 1 <1
$$
shows that the map
$$
S^{n-1} \rightarrow [0,1]; \qquad y \mapsto \int_{S^{n-1}}d(x,y)^p \dd \mu(x)
$$
is strictly less than $1$ everywhere. Therefore, we have
\be \label{eq:kisebb-egynel}
\int_{S^{n-1}} \int_{S^{n-1}} d(x,y)^p \dd \mu(x) \dd \nu(y) <1
\ee
for any Borel probability measure $\nu.$
The map $(x,y)\mapsto d(x,y)^p$ is bounded and both $\mu$ and $\nu$ are probability measures, hence \emph{Fubini's theorem} can be applied to show that the integral on the left hand side of \eqref{eq:kisebb-egynel} is equal to
$$
\int_{S^{n-1} \times S^{n-1}} d(x,y)^p \dd \ler{\mu \times \nu}(x,y).
$$
The measure $\mu \times \nu$ is clearly a coupling of $\mu$ and $\nu,$ hence by the definition of the Wasserstein distance (see eq. \eqref{eq:wasser_def}) we have
$$
W_p^p\ler{\mu, \nu}=\inf_{\pi \in \Pi(\mu, \nu)} \int_{S^{n-1} \times S^{n-1}} d^p (x,y) \mathrm{d}\pi(x,y)
$$
$$
\leq
\int_{S^{n-1} \times S^{n-1}} d(x,y)^p \dd \ler{\mu \times \nu}(x,y).
$$
So, we deduced that $W_p^p\ler{\mu, \nu}<1$ which means that $W_p\ler{\mu, \nu}<1.$ The measure $\nu \in P\ler{S^{n-1}}$ was arbitrary, hence the proof is done.
\end{proof}

The following result may be considered as a first step in the direction of a Banach-Stone-type result on the structure of the Wasserstein isometries of probability measures on unit spheres.

\begin{theorem} \label{thm:main}
Let $\phi: P\ler{S^{n-1}} \rightarrow P\ler{S^{n-1}}$ be a (not necessarily surjective) Wasserstein isometry, that is, a map satisfying
$$
W_p \ler{\phi(\mu),\phi(\nu)}=W_p \ler{\mu,\nu} \qquad \ler{\mu, \nu \in P\ler{S^{n-1}}}.
$$
Then there exists an isometry $T: S^{n-1} \rightarrow S^{n-1}$ such that
$$
\phi\ler{\delta_x}=T_\# \delta_x, \text{ that is, } \phi\ler{\delta_x}=\delta_{T(x)} \qquad \ler{x \in S^{n-1}}. 
$$
\end{theorem}

\begin{proof}
Let $x \in S^{n-1}$ be arbitrary. By Claim \ref{claim:dirac-char}, there exists a $\nu \in P\ler{S^{n-1}}$ such that $W_p\ler{\delta_x, \nu}=1.$ By assumption, $W_p\ler{\phi\ler{\delta_x}, \phi\ler{\nu}}=1.$ By Claim \ref{claim:dirac-char}, this means that $\phi\ler{\delta_x}$ is a Dirac measure. So, $\phi$ sends Dirac measures to Dirac measures. That is, there exists a map $T: S^{n-1} \rightarrow S^{n-1}$ such that $\phi\ler{\delta_x}=\delta_{T(x)}$ holds for all $x \in S^{n-1}.$ We have to show that $T$ is an isometry. But this is clear, because $W_p\ler{\delta_x, \delta_y}=d(x,y).$ Indeed, by this elementary fact we have
$$
d(x,y)=W_p\ler{\delta_{x},\delta_{y}}=W_p\ler{\phi\ler{\delta_{x}},\phi\ler{\delta_{y}}}
$$
$$
=W_p\ler{\delta_{T(x)},\delta_{T(y)}}=d\ler{T(x), T(y)}
$$
for every $x,y \in S^{n-1}.$ The proof is done.
\end{proof}

\subsection*{Final remarks}
Let us emphasize that we did not assume the surjectiviy of the isometries in our previous arguments.
\par
Naturally our most concrete future plan is to discover wheter the measure spaces on the unit balls are isometrically rigid, or we also have some non-trivial isometries (let alone exotic isometries). We believe that the answer depends on the dimension $n$ and on the value of the parameter $p,$ as well.

\subsection*{Acknowledgement}
The author is grateful to Gy\"orgy P\'al Geh\'er and Tam\'as Titkos for drawing his attention to some of the works listed in the Bibliography, and for useful discussions. The author is also grateful to the anonymous referee for many valuable comments and suggestions that helped to improve the presentation of the paper substantially.

\end{document}